\documentclass[a4paper]{article}
\usepackage[dutch,english]{babel}
\usepackage{amsmath}
\usepackage{amsfonts}
\usepackage{amsthm}
\usepackage{algorithm}
\usepackage{algpseudocode}
\usepackage{picture}
\usepackage{hyperref}

\newcommand{\mcc}{{\mathbb{C}}}
\newcommand{\mff}{{\mathbb{F}}}
\newcommand{\mqq}{{\mathbb{Q}}}
\newcommand{\mrr}{{\mathbb{R}}}
\newcommand{\mzz}{{\mathbb{Z}}}

\newcommand{\msetb}[1]{{\left\{{#1}\right\}}}				
\newcommand{\msetf}[1]{{\left\lfloor{#1}\right\rfloor}}		
\newcommand{\msetp}[1]{{\left({#1}\right)}}					
\newcommand{\msets}[1]{{\left[{#1}\right]}}					
\newcommand{\msetv}[1]{{\left|{#1}\right|}}					

\theoremstyle{plain}
\newtheorem{lemma}{Lemma}
\newtheorem{proposition}[lemma]{Proposition}
\newtheorem{theorem}[lemma]{Theorem}

\newtheorem*{question}{Question}

\hypersetup{
	pdftitle={Elliptic curves maximal over extensions of finite base fields},
	pdfauthor={A.S.I. Anema},
}

\title{Elliptic curves maximal over extensions of finite base fields}
\author{A.S.I. Anema}
\date{}

\begin{document}

\maketitle

\begin{abstract}
Given an elliptic curve $E$ over a finite field $\mathbb{F}_q$ we study the
finite extensions $\mathbb{F}_{q^n}$ of $\mathbb{F}_q$ such that the number of
$\mathbb{F}_{q^n}$-rational points on $E$ attains the Hasse upper bound. We
obtain an upper bound on the degree $n$ for $E$ ordinary using an estimate for
linear forms in logarithms, which allows us to compute the pairs of isogeny
classes of such curves and degree $n$ for small $q$. Using a consequence of
Schmidt's Subspace Theorem, we improve the upper bound to $n\leq 11$ for
sufficiently large $q$. We also show that there are infinitely many isogeny
classes of ordinary elliptic curves with $n=3$.
\end{abstract}

\section{Introduction}

Let $E$ be an elliptic curve over $\mff_q$. Recall the well-known Hasse bound on
the number of points on an elliptic curve
\begin{displaymath}
\msetv{\msetv{E\msetp{\mff_{q^n}}}-q^n-1}
\leq
\msetf{2\sqrt{q}^n},
\end{displaymath}
see for example~\cite[Theorem V.1.1]{silverman1985} or \cite[Theorem~5.1.1]
{soomro2013}. If $E$ attains the Hasse upper bound over some finite extension,
that is
\begin{displaymath}
\msetv{E\msetp{\mff_{q^n}}}
=q^n+1+\msetf{2\sqrt{q}^n}
\end{displaymath}
for some $n$, then we say $E$ is \emph{maximal} over $\mff_{q^n}$. We are
interested in:

\begin{question}
Let $E$ be an elliptic curve over $\mff_q$. Is $E$ maximal over some finite
extension of $\mff_q$?
\end{question}

This question is studied and partially answered by Doetjes in~\cite
{doetjes2009}. He shows that every elliptic curve over $\mff_2$ is maximal over
some extension, that elliptic curves over $\mff_3$ in five isogeny classes are
maximal over some extension, that elliptic curves over $\mff_3$ in the remaining
two isogeny classes are not maximal over extensions of low degree, and that
elliptic curves over $\mff_q$ with $q$ a square are maximal over some extension
in precisely three cases.

Our first result is summarized as:

\begin{theorem}
\label{mc:thm:main}
Let $E$ be an elliptic curve over $\mff_q$ and $a_1=q+1-\msetv{E\msetp{\mff_q}
}$.
\begin{enumerate}
\item If $E$ is supersingular, that is $a_1\in\msetb{0,\pm\sqrt{q},\pm\sqrt{2q},
\pm\sqrt{3q},\pm 2\sqrt{q}}$, then $E$ is maximal over infinitely many
extensions of $\mff_q$ except when $a_1\in\msetb{-\sqrt{q},2\sqrt{q}}$. In these
exceptional cases extensions over which $E$ is maximal do not exist.
\item If $E$ is ordinary, that is $\gcd\msetp{a_1,q}=1$, then there are at most
finitely many extensions of $\mff_q$ over which $E$ is maximal. Furthermore if
$q$ is a square, then such extensions do not exist.
\end{enumerate}
\end{theorem}

We prove the first part of the theorem in Section~\ref{mc:sec:ssc}. The second
part we treat in Section~\ref{mc:sec:ordc}. There we also give an explicit bound
on the degree of the extension and list the pairs $q,a_1$ with $q<1000$
corresponding to ordinary elliptic curves over $\mff_q$ maximal over some finite
extension. In Subsection~\ref{mc:ssec:ubcard} we show that the degree of the
extension is at most 11 for sufficiently large $q$.

Our second result is:

\begin{theorem}
\label{mc:thm:max3prime}
For infinitely many primes $p$ there exists an elliptic curve $E$ over $\mff_p$
such that $E$ is maximal over $\mff_{p^3}$.
\end{theorem}

This confirms an observation made by Soomro in \cite[Section~2.7]{soomro2013}
as well as our computations in Subsection~\ref{mc:ssec:algorithm}. We prove the
theorem in Section~\ref{mc:sec:max3}.

Notice that the property of $E$ to be maximal over $\mff_{q^n}$ depends only on
the isogeny class of $E$, because isogenous elliptic curves over a finite field
have the same number of points, see \cite[Lemma~15.1]{cassels1966}. The isogeny
classes of elliptic curves over $\mff_q$ correspond to integers $a_1$ such that
$\msetv{a_1}\leq 2\sqrt{q}$ and some additional conditions, see \cite
[Theorem~4.1]{waterhouse1969}. Define the integers $a_n$ as
\begin{displaymath}
a_n=q^n+1-\msetv{E\msetp{\mff_{q^n}}}.
\end{displaymath}
If $\alpha$ is an eigenvalue of Frobenius, that is a root of the polynomial
$X^2-a_1X+q$, then $a_n=\alpha^n+\bar{\alpha}^n$ with $\bar{\alpha}$ the
conjugate of $\alpha$, see \cite[Section~V.2]{silverman1985}. So, the $a_n$'s
satisfy the recurrence relation
\begin{displaymath}
a_{n+1}=a_1a_n-qa_{n-1}
\end{displaymath}
for $n$ a positive integer and $a_0=2$. Hence we reduced our question to:

\begin{question}
Let $q$ be a prime power and $a_1$ an integer such that $\msetv{a_1}\leq 2
\sqrt{q}$. Is there a positive integer $n$ such that $-a_n=\msetf{2\sqrt{q}^n}
$?
\end{question}

In this chapter $q,a_1$ are integers with $q\geq 2$ and $\msetv{a_1}\leq 2\sqrt
{q}$, $\alpha$ is a root of $X^2-a_1X+q$ and $\beta=\frac{\alpha}{\sqrt{q}}$.
Fix an embedding $\mqq\msetp{\sqrt{q},\alpha}\rightarrow\mcc$ such that $\sqrt
{q}>0$ and $\alpha$ lies in the upper half-plane, that is $\arg\msetp{\alpha}\in
\msets{0,\pi}$.

If $\beta$ is a root of unity, then the pair $q,a_1$ is called \emph
{supersingular}, otherwise the pair is called \emph{ordinary}. This definition
agrees with the one for elliptic curves whenever the pair $q,a_1$ corresponds
to an isogeny class of elliptic curves, see again \cite[Theorem~4.1]
{waterhouse1969}.

The answer to the question is divided into two cases, namely the supersingular
case and the ordinary case.

\noindent{\bfseries Acknowledgement.} This research was performed by the author
at the University of Groningen for his PhD thesis and was financially supported
by \emph{Discrete, Interactive and Algorithmic Mathematics, Algebra and Number
Theory} (DIAMANT); a mathematics cluster funded by the \emph{Netherlands
Organisation for Scientific Research} (NWO).

\section{Supersingular case}
\label{mc:sec:ssc}

The first part of Theorem~\ref{mc:thm:main} follows directly from:

\begin{proposition}
\label{mc:prop:qscm}
Let $q,a_1$ be integers with $q\geq 2$ and $\msetv{a_1}\leq 2\sqrt{q}$. If the
pair $q,a_1$ is supersingular, then $-a_n=\msetf{2\sqrt{q}^n}$ for some positive
integer $n$ if and only if
\begin{displaymath}
a_1\in\msetb{0,\sqrt{q},\pm\sqrt{2q},\pm\sqrt{3q},-2\sqrt{q}}.
\end{displaymath}
Moreover if such an integer $n$ exists, then there exist infinitely many.
\end{proposition}

The proposition above extends the result for $\mff_q$ with $q$ a square
presented in \cite[Chapter~5]{doetjes2009} to arbitrary $q\geq 2$. The new proof
uses the following results:

\begin{lemma}
\label{mc:lem:bierou}
If $\beta$ is a root of the polynomial $X^2-\frac{a_1}{\sqrt{q}}X+1$ with $q,
a_1$ integers and $q$ non-zero, then $\beta$ is a root of unity if and only if
\begin{displaymath}
a_1
\in
\msetb{0,\pm\sqrt{q},\pm\sqrt{2q},\pm\sqrt{3q},\pm 2\sqrt{q}}.
\end{displaymath}
\end{lemma}
\begin{proof}
Suppose that $\beta$ is a primitive root of unity of order $n$. Let $\varphi$
denote Euler's function, then $\msets{\mqq\msetp{\beta}:\mqq}=\varphi
\msetp{n}$. Since $\msets{\mqq\msetp{\sqrt{q},\beta}:\mqq}\in\msetb{1,2,4}$,
the same is true for $\msets{\mqq\msetp{\beta}:\mqq}$. The cyclotomic
polynomials of degree dividing 4 are listed in Table~\ref{mc:tab:eprud4}.
Evaluate $X^2-\frac{a_1}{\sqrt{q}}X+1$ in a primitive root of unity $\zeta_n$
of order $n$ for $n=1,2,3,4,6$ to obtain $a_1=2\sqrt{q},-2\sqrt{q},-\sqrt{q},0,
\sqrt{q}$ respectively. Notice that $\beta$ is also a root of $X^4+\msetp{2-
\frac{a_1^2}{q}}X^2+1$, and this polynomial and the cyclotomic polynomial both
have degree 4 for $n=5,8,10,12$. This implies that $a_1=\pm\sqrt{2q},\pm\sqrt
{3q}$ for $n=8,12$ respectively, and that the cases $n=5,10$ are impossible.
Hence $a_1$ is as desired.

Assume that
\begin{displaymath}
a_1\in\msetb{0,\pm\sqrt{q},\pm\sqrt{2q},\pm\sqrt{3q},\pm 2\sqrt{q}}.
\end{displaymath}
If $a_1=\pm 2\sqrt{q}$, then $X^2-\frac{a_1}{\sqrt{q}}X+1=\msetp{X\mp 1}^2$,
that is $\beta$ is a root of unity. Since $\beta$ is a root of $X^2-\frac{a_1}
{\sqrt{q}}X+1$, $\beta$ is also a root of $X^4+\msetp{2-\frac{a_1^2}{q}}X^2+1$.
If $a_1\neq\pm 2\sqrt{q}$, then one of both polynomials is listed Table~\ref
{mc:tab:eprud4}, that is $\beta$ is a root of unity. Hence in either case
$\beta$ is a root of unity.
\end{proof}

\begin{table}
\begin{center}
\caption{The list of all cyclotomic polynomials $\Phi_n$ of degree $d$ dividing
4. Recall that $\varphi\msetp{n}=\msets{\mqq\msetp{\zeta_n}:\mqq}=d$.}
\begin{tabular}{ccr}
$d$ & $n$ & \multicolumn{1}{c}{$\Phi_n$} \\
\hline
1 &  1 & $X-1$ \\
  &  2 & $X+1$ \\
\hline
2 &  3 & $X^2+X+1$ \\
  &  4 & $X^2+1$ \\
  &  6 & $X^2-X+1$ \\
\hline
4 &  5 & $X^4+X^3+X^2+X+1$ \\
  &  8 & $X^4+1$ \\
  & 10 & $X^4-X^3+X^2-X+1$ \\
  & 12 & $X^4-X^2+1$
\end{tabular}
\label{mc:tab:eprud4}
\end{center}
\end{table}

\begin{lemma}
\label{mc:lem:translation}
Let $q,a_1$ be integers with $q$ positive and $\msetv{a_1}\leq 2\sqrt{q}$. If
$n$ is a positive integer, then
\begin{displaymath}
-a_n=\msetf{2\sqrt{q}^n}
\quad\Longleftrightarrow\quad
\msetv{\beta^n+1}<\frac{1}{\sqrt[4]{q}^n}.
\end{displaymath}
\end{lemma}
\begin{proof}
Notice that $-a_n=\msetf{2\sqrt{q}^n}$ is equivalent to $-a_n\leq 2\sqrt{q}^n<
-a_n+1$, which is the same as $0\leq a_n+2\sqrt{q}^n<1$. Since $\msetv{a_n}\leq
2\sqrt{q}^n$ implies $0\leq a_n+2\sqrt{q}^n$, in fact $-a_n=\msetf{2\sqrt{q}^n}$
if and only if $\msetv{a_n+2\sqrt{q}^n}<1$.

Recall that $a_n=\alpha^n+\bar{\alpha}^n$ and $\msetv{\alpha}=\sqrt{q}$ and
$\beta=\frac{\alpha}{\msetv{\alpha}}$. Observe that
\begin{displaymath}
a_n+2\sqrt{q}^n
=\alpha^n+\bar{\alpha}^n+2\sqrt{q}^n
=\bar{\alpha}^n\msetp{\beta^{2n}+1+2\beta^n}
=\bar{\alpha}^n\msetp{\beta^n+1}^2.
\end{displaymath}
Substitute this relation in the last inequality to complete the proof.
\end{proof}

\begin{proof}[Proof of Proposition~\ref{mc:prop:qscm}]
Suppose that $\msetv{\beta^n+1}<\frac{1}{\sqrt[4]{q}^n}$ for some positive
integer $n$ and $\beta^m+1\neq 0$ for all integers $m$. Recall that $\beta$ is
a root of $X^2-\frac{a_1}{\sqrt{q}}X+1$ and by assumption $\beta$ is also a
root of unity. Thus the order of $\beta$ is odd. According to Lemma~\ref
{mc:lem:bierou} and its proof $\beta$ has order 1 or 3. If the order is 1, then
$\msetv{\beta^m+1}=2$ for all integers $m$. If the order is 3, then $\msetv{
\beta^m+1}\geq 1$ for all integers $m$. In either case this contradicts $\msetv
{\beta^n+1}<\frac{1}{\sqrt[4]{q}^n}$. Hence for $n$ a positive integer
\begin{displaymath}
\msetv{\beta^n+1}<\frac{1}{\sqrt[4]{q}^n}
\quad\Longleftrightarrow\quad
\beta^n+1=0.
\end{displaymath}

Lemma~\ref{mc:lem:bierou} implies that $\beta^n+1=0$ for some positive integer
$n$ if and only if the order of $\beta$ is even if and only if $a_1\in\msetb{0,
\sqrt{q},\pm\sqrt{2q},\pm\sqrt{3q},-2\sqrt{q}}$.

The proposition follows from Lemma~\ref{mc:lem:translation}.
\end{proof}

\section{Ordinary case}
\label{mc:sec:ordc}

The first result restricting the possible values of $q$ and $n$ in this case is:

\begin{proposition}
\label{mc:prop:oqnsno}
Let $q,a_1$ be integers with $q\geq 2$ and $\msetv{a_1}\leq 2\sqrt{q}$. If the
pair $q,a_1$ is ordinary and $-a_n=\msetf{2\sqrt{q}^n}$ for some positive
integer $n$, then $q$ is not a square and $n$ is odd.
\end{proposition}

\begin{proof}
Assume that $-a_n=\msetf{2\sqrt{q}^n}$ for some positive integer $n$. Recall
that $\beta=\frac{\alpha}{\msetv{\alpha}}$. If $q$ is a square or $n$ is even,
then $\msetf{2\sqrt{q}^n}=2\sqrt{q}^n$, that is $\beta^n+1=0$ (see Lemma~\ref
{mc:lem:translation}). However by assumption $\beta$ is not a root of unity.
\end{proof}

\subsection{Upper bound on the degree}

Given an ordinary pair $q,a_1$ we derive an upper bound on the $n$'s such that
$-a_n=\msetf{2\sqrt{q}^n}$ using an estimate for linear forms in two logarithms
from~\cite{laurent1995}.

\begin{proposition}
\label{mc:prop:laurentbound}
For every integer $q\geq 2$ let $N_q$ be the unique zero of
\begin{displaymath}
n\longmapsto
\frac{n}{4}\log\msetp{q}-8.87\msetp{10.98\pi+\frac{1}{2}\log\msetp{q}}\msetp
{2\log\msetp{n}+3.27}^2-\log\msetp{\frac{\pi}{3}}
\end{displaymath}
larger than $8007$.
\begin{enumerate}
\item The sequence $\msetb{N_q}_{q\geq 2}$ decreases monotonically.
\item If the pair of integers $q,a_1$ with $q\geq 2$ and $\msetv{a_1}\leq 2
\sqrt{q}$ is ordinary and $-a_n=\msetf{2\sqrt{q} ^n}$ for some $n$, then $n<
N_q$.
\end{enumerate}
\end{proposition}

We computed the value of $N_q$ for several $q$ and list them in Table~\ref
{mc:tbl:qfnq}. In the case $q=3$ Doetjes mentioned~\cite[p. 25]{doetjes2009}
that for $a_1=-2$ and $a_1=1$ there are no $n<10^6$ such that $-a_n=\msetf{2
\sqrt{q}^n}$, and he expected that such $n$ do not exist at all. Since his
argument extends to $n<2998887$, our upper bound on $n$ shows that his
observation is correct.

\begin{table}
\begin{center}
\caption{The value of $\msetf{N_q}$ for various $q$.}
\begin{tabular}{cc|cc}
$q$ & $\msetf{N_q}$ & $q$ & $\msetf{N_q}$ \\
\hline
2 & 1840001 & $10^3$ & 142072 \\
3 & 1093182 & $10^4$ & 104910 \\
10 & 475174 & $10^5$ & 83424 \\
$10^2$ & 220290 & $10^6$ & 69510
\end{tabular}
\label{mc:tbl:qfnq}
\end{center}
\end{table}

We denote the principal value of the argument and the complex logarithm by
$\arg$ and $\log$ respectively.

\begin{lemma}
\label{mc:lem:logupperbound}
Let $q,a_1$ be integers with $q$ positive and $\msetv{a_1}\leq 2\sqrt{q}$. If
$n$ is a positive integer such that $-a_n=\msetf{2\sqrt{q}^n}$, then
\begin{displaymath}
\msetv{m\pi+n\arg\msetp{\beta}}
=\msetv{\arg\msetp{-\beta^n}}
<\frac{\pi}{3}\frac{1}{\sqrt[4]{q}^n}
\end{displaymath}
for some odd integer $m$ such that $\msetv{m}\leq n$.
\end{lemma}
\begin{proof}
Assume that $-a_n=\msetf{2\sqrt{q}^n}$ for some positive integer $n$. Since
$\msetv{\beta}=1$ by construction and $\msetv{\beta^n+1}<\frac{1}{\sqrt[4]{q}
^n}<1$ by Lemma~\ref{mc:lem:translation}, $\msetv{\arg\msetp{-\beta^n}}<\frac
{\pi}{3}$. Use $\msetv{\sin\msetp{\phi}}\geq\frac{3}{\pi}\msetv{\phi}$ for
$\msetv{\phi}\leq\frac{\pi}{6}$ and $\frac{1}{2}\msetv{z-1}=\msetv{\sin\msetp
{\frac{1}{2}\arg\msetp{z}}}$ for $\msetv{z}=1$ to obtain
\begin{displaymath}
\msetv{\arg\msetp{-\beta^n}}
<\frac{\pi}{3}\frac{1}{\sqrt[4]{q}^n}.
\end{displaymath}
Notice that
\begin{displaymath}
\arg\msetp{-\beta^n}
=\arg\msetp{-1}+n\arg\msetp{\beta}+2\pi k
=\msetp{2k+1}\pi+n\arg\msetp{\beta}
\end{displaymath}
for some integer $k$. Define $m=2k+1$. Since $\msetv{\arg\msetp{\beta}}\leq\pi$
and $\msetv{\arg\msetp{-\beta^n}}<\frac{\pi}{3}$,
\begin{displaymath}
\msetv{m}\pi
=\msetv{\arg\msetp{-\beta^n}-n\arg\msetp{\beta}}
\leq\msetv{\arg\msetp{-\beta^n}}+n\msetv{\arg\msetp{\beta}}
<\msetp{n+\frac{1}{3}}\pi,
\end{displaymath}
that is $\msetv{m}\leq n$ as $m,n$ are integers.
\end{proof}

The \emph{logarithmic height} of an algebraic number $\beta$ is defined as
\begin{displaymath}
\frac{1}{n}\msetp{\log\msetv{b}+\sum_{i=1}^n\log\max\msetb{1,\msetv{\beta_i}}}
\end{displaymath}
with $b\prod_{i=1}^n\msetp{X-\beta_i}$ the minimal polynomial of $\beta$ over
$\mzz$.

\begin{lemma}
\label{mc:lem:laurent}
Let $\beta$ be an algebraic number of absolute value one. If $\beta$ is not a
root of unity, $n$ a positive integer and $m$ a non-zero integer such that
$\msetv{m}\leq n$, then $\log\msetv{m\pi+n\arg\msetp{\beta}}$ is at least
\begin{displaymath}
-8.87\msetp{10.98\pi+dl}
\max\msetb{17,\frac{\sqrt{d}}{10},d\log\msetp{n}-0.88d+5.03}^2,
\end{displaymath}
where $l$ is an upper bound on the logarithmic height of $\beta$ and $d=\frac
{1}{2}\msets{\mqq\msetp{\beta}:\mqq}$.
\end{lemma}
\begin{proof}
Since the logarithmic heights of $\beta$ and $\bar{\beta}$ are equal and $\arg
\msetp{\bar{\beta}}=-\arg\msetp{\beta}$, replace $\beta$ by $\bar{\beta}$ and
$m$ by $-m$ as necessary to reduce to the case of negative $m$.

Let $a$ and $H$ be as in \cite[Th\'eor\`eme~3]{laurent1995}. Observe that
\begin{displaymath}
20\leq
a\leq 10.98\pi+dl
\end{displaymath}
and using $20\leq a$ and $\msetv{m}\leq n$ that
\begin{displaymath}
H
\leq
\max\msetb{17,\frac{\sqrt{d}}{10},
d\log\msetp{n}-0.88d+5.03}.
\end{displaymath}
Since $\msetv{\beta}=1$ and $\beta$ is not a root of unity, apply \cite
[Th\'eor\`eme~3]{laurent1995} to $\msetv{m}\pi i-n\log\msetp{\beta}$ and use
$\log\msetp{\beta}=\arg\msetp{\beta}i$ to obtain the desired lower bound.
\end{proof}

\begin{lemma}
\label{mc:lem:ordbetamp}
Let $q,a_1$ be integers with $q\geq 2$ and $\msetv{a_1}\leq 2\sqrt{q}$. If the
pair $q,a_1$ is ordinary and $q$ is not a square, then the minimal polynomial of
$\beta$ over $\mqq$ is
\begin{displaymath}
X^4+\msetp{2-\frac{a_1^2}{q}}X^2+1.
\end{displaymath}
\end{lemma}
\begin{proof}
Since $\beta$ is a root of $X^2-\frac{a_1}{\sqrt{q}}X+1$, it is also a root of
the polynomial above. Proposition~\ref{mc:prop:qscm} gives $a_1\neq 0,\pm 2
\sqrt{q}$, because $\beta$ is not a root of unity. Thus $\sqrt{q}\in\mqq\msetp
{\beta}$ and $\mrr\msetp{\beta}=\mcc$. Hence $\msets{\mqq\msetp{\beta}:\mqq}=4$,
that is the degree 4 polynomial is irreducible.
\end{proof}

\begin{proof}[Proof of Proposition~\ref{mc:prop:laurentbound}]
Define the function $f:\mrr_{\geq 1}\times\mrr_{\geq 1}
\rightarrow\mrr$ as
\begin{displaymath}
\msetp{q,n}\longmapsto
\frac{n}{4}\log\msetp{q}-8.87\msetp{10.98\pi+\frac{1}{2}\log\msetp{q}}\msetp{
2\log\msetp{n}+3.27}^2-\log\msetp{\frac{\pi}{3}}.
\end{displaymath}
Observe that $f\msetp{1,n}<0$ for all $n\geq 1$. The function $\mrr_{\geq 1}
\rightarrow \mrr$
\begin{displaymath}
n\longmapsto q\cdot \frac{\partial f}{\partial q}\msetp{q,n}
=\frac{n}{4}-\frac{8.87}{2}\msetp{2\log\msetp{n}+3.27}^2
\end{displaymath}
is independent of $q$, strictly convex, has a unique minimum $1243<n_1<1244$
and has a unique zero $8007<n_2<8008$ such that $n_2>n_1$. Therefore $\frac{
\partial f}{\partial q}\msetp{q,\msetf{n_2}}<0$ for all $q\geq 1$, and as a
result $f\msetp{q,\msetf{n_2}}<0$ for all $q\geq 1$. On the other hand $n
\mapsto f\msetp{q,n}$ is also strictly convex, because (for $q,n\geq 1$)
\begin{displaymath}
\frac{\partial^2 f}{\partial n^2}\msetp{q,n}=
35.48\msetp{10.98\pi+\frac{1}{2}\log\msetp{q}}\frac{2\log\msetp{n}+1.27}{n^2}
 >0.
\end{displaymath}
Moreover if $q>1$, then $f\msetp{q,n}>0$ for sufficiently large $n$. Combined
with $f\msetp{q,\msetf{n_2}}<0$ this shows that for all $q>1$ there exists a
unique $N_q>\msetf{n_2}$ such that $f\msetp{q,N_q}=0$.

Since $q\cdot\frac{\partial f}{\partial q}\msetp{q,n}=c_n>0$ for all $n>\msetf
{n_2}$, $f\msetp{q^\prime,N_q}>f\msetp{q,N_q}=0$ for all $q^\prime>q$, which
implies that $N_{q^\prime}<N_q$ for all $q^\prime>q$. Hence the first part of
the proposition follows.

Assume that the pair $q,a_1$ is ordinary and $-a_n=\msetf{2\sqrt{q}^n}$ for
some integer $n$. Lemma~\ref{mc:lem:logupperbound} gives
\begin{displaymath}
\msetv{m\pi+n\arg\msetp{\beta}}
<\frac{\pi}{3}\frac{1}{\sqrt[4]{q}^n}
\end{displaymath}
for some odd integer $m$ such that $\msetv{m}\leq n$. The integer $q$ is not a
square by Proposition~\ref{mc:prop:oqnsno} and the minimal polynomial of
$\beta$ over $\mzz$ has degree 4 and divides
\begin{displaymath}
qX^4+\msetp{2q-a_1^2}X^2+q
\end{displaymath}
by Lemma~\ref{mc:lem:ordbetamp}, so that $\msets{\mqq\msetp{\beta}:\mqq}=4$.
Since $\msetv{\beta}= 1$ and $\beta$ is not a root of unity, $\beta$, $\bar
{\beta}$, $-\beta$ and $-\bar{\beta}$ are the distinct roots of this polynomial
so that the logarithmic height of $\beta$ is at most $\frac{1}{4}\log\msetp{q}
$. Lemma~\ref{mc:lem:laurent} says
\begin{displaymath}
\log\msetv{m\pi+n\arg\msetp{\beta}}
\geq
-8.87\msetp{10.98\pi+\frac{1}{2}\log\msetp{q}}\max\msetb{17,
2\log\msetp{n}+3.27}^2.
\end{displaymath}
Let $n_0$ be such that $17=2\log\msetp{n_0}+3.27$, that is $n_0=e^{6.865}
\approx 958.1$. If $n\geq N_q$, then $n>\msetf{n_2}>n_0$ and so $f\msetp{q,n}
<0$ by the upper and lower bounds on $\msetv{m\pi+n\arg\msetp{\beta}}$, which
contradicts $f\msetp{q,n}\geq 0$ for all $n\geq N_q$. This proves the second
part of the proposition.
\end{proof}

\subsection{Computing maximal triples}
\label{mc:ssec:algorithm}

Given an ordinary pair $q,a_1$ the upper bound in Proposition~\ref
{mc:prop:laurentbound} reduces the problem of determining the $n>1$ such that
$-a_n=\msetf{2\sqrt{q}^n}$ to a finite computation. An efficient method to
compute such $n$ is described in \cite[Section~6.1]{doetjes2009}: $n$ is
essentially the denominator of a convergent of $\frac{\arg\msetp{\beta}}{\pi}$.
We extend \cite[Stelling~6.8]{doetjes2009} in order to take into account
numerical errors.

\begin{proposition}
\label{mc:prop:apprconv}
Let $q,a_1$ be integers with $q\geq 2$ and $a_1=2\sqrt{q}\cos\msetp{\theta}$
for some $\theta\in\msets{0,\pi}$ and $x\in\mrr$ such that for some positive
integer $N$
\begin{displaymath}
\msetv{x-\frac{\theta}{\pi}}
\leq
\frac{1}{2N^2}\cdot\left\{
\begin{array}{ll}
1-\frac{2}{3}\frac{13}{\sqrt[4]{2}^{13}}
& \text{if } q=2, \\
1-\frac{2}{3}\frac{3}{\sqrt[4]{q}^3}
& \text{if } q\geq 3.
\end{array}
\right.
\end{displaymath}
If $-a_n=\msetf{2\sqrt{q}^n}$ for some odd integer $3\leq n\leq N$ and either
$q\geq 3$ or $n\geq 13$, then $\frac{m}{n}$ is a convergent of $x$ for some odd
$m$.
\end{proposition}

Propositions~\ref{mc:prop:oqnsno} and~\ref{mc:prop:apprconv} together with $x=
\frac{\theta}{\pi}$ imply \cite[Stelling~6.8]{doetjes2009}.

\begin{proof}
Assume that $-a_n=\msetf{2\sqrt{q}^n}$ for some $n$. Since $\arg\msetp{\beta}=
\theta$ by the choice of $\beta$, Lemma~\ref{mc:lem:logupperbound} implies that
\begin{displaymath}
\msetv{\frac{m}{n}-\frac{\theta}{\pi}}
<\frac{1}{3}\frac{1}{n\sqrt[4]{q}^n}
\end{displaymath}
for some odd integer $m$ such that $\msetv{m}\leq n$. If $x\in\mrr$ such that
\begin{displaymath}
\msetv{x-\frac{\theta}{\pi}}
\leq
\frac{1}{2n^2}-\frac{1}{3}\frac{1}{n\sqrt[4]{q}^n},
\end{displaymath}
then $\msetv{x-\frac{m}{n}}<\frac{1}{2n^2}$ so that $\frac{m}{n}$ is a
convergent of $x$ by \cite[Theorem~184]{hardy1975}.

Define the function $f:\mrr\rightarrow\mrr$ as
\begin{displaymath}
f\msetp{n}
=1-\frac{2}{3}\frac{n}{\sqrt[4]{q}^n}.
\end{displaymath}
It has a global minimum at $n_0=\frac{4}{\log\msetp{q}}$. Observe that $f\msetp
{n_0}$ is positive except for $q=2$. Consider the following cases:
\begin{itemize}
\item If $q=2$, then $n=13$ is the first integer for which $f\msetp{n}$ is
positive.
\item If $q=3$, then $3<n_0<4$ and $f\msetp{3}<f\msetp{5}$.
\item If $q\geq 4$, then $n_0<3$.
\end{itemize}
Since $3\leq n\leq N$ is odd and either $q\geq 3$ or $n\geq 13$,
\begin{displaymath}
\frac{f\msetp{n}}{2n^2}
\geq \frac{f\msetp{n}}{2N^2}
\geq \frac{1}{2N^2}\left\{
\begin{array}{ll}
f\msetp{13} & \text{if } q=2, \\
f\msetp{3} & \text{if } q\geq 3. \\
\end{array}
\right.
\geq \msetv{x-\frac{\theta}{\pi}}.
\end{displaymath}
Hence $\frac{m}{n}$ is a convergent of $x$.
\end{proof}

Beware of applying the above proposition. If $-a_n=\msetf{2\sqrt{q}^n}$ for
some $n$, then $\frac{m}{n}$ is equal to a convergent of $\frac{\theta}{\pi}$
according to the proposition, but $m$ and $n$ need not be relative prime.
However let $d=\gcd\msetp{m,n}$ and $\tilde{n}=\frac{n}{d}$, then
\begin{displaymath}
\arg\msetp{-\beta^{\tilde{n}}}
=\frac{1}{d}\arg\msetp{-\beta^n}
\end{displaymath}
and
\begin{displaymath}
\frac{1}{2}\msetv{\beta^{\tilde{n}}+1}
=\msetv{\sin\msetp{\frac{1}{2}\arg\msetp{-\beta^{\tilde{n}}}}}
\leq\msetv{\sin\msetp{\frac{1}{2}\arg\msetp{-\beta^n}}}
=\frac{1}{2}\msetv{\beta^n+1},
\end{displaymath}
so that $-a_{\tilde{n}}=\msetf{2\sqrt{q}^{\tilde{n}}}$ by Lemma~\ref
{mc:lem:translation} for $n$ and $\tilde{n}$.

\begin{algorithm}
\caption{The procedure \textsc{MaximalCurves} takes as input integers $q,a_1$
with $q\geq 2$ and $\msetv{a_1}\leq 2\sqrt{q}$ such that the pair is ordinary
and outputs the $n$'s with $n>1$ such that $-a_n=\msetf{2\sqrt{q}^n}$. The
function \textsc{MaximalDegree}($q$) returns the upper bound on $n$ from
Proposition~\ref{mc:prop:laurentbound}, the function \textsc{Convergents}($x$,
$N$) computes the convergents of $x$ with denominator at most $N$ and the
function \textsc{IsSolution}($q$,$a_1$,$n$) checks $-a_n=\msetf{2\sqrt{q}^n}$.}
\label{mc:alg:maxcurves}
\begin{algorithmic}[1]
\Procedure{MaximalCurves}{$q$, $a_1$}
\If{$q$ not square}
	\State $\theta\gets$ \Call{ArcCos}{$\frac{a_1}{2\sqrt{q}}$}
	\State $N\gets$ \Call{MaximalDegree}{$q$}
	\State $C\gets$ \Call{Convergents}{$\frac{\theta}{\pi}$, $N$}
	\If{$q=2$}
		\ForAll{$n\in\msetb{3,5,7,9,11}$}
			\If{\Call{IsSolution}{$q$, $a_1$, $n$}}
				\State \textbf{print} $n$
			\EndIf
		\EndFor
	\EndIf
	\ForAll{$\frac{m}{n}\in C:m\text{ odd}, n\text{ odd}$}
		\State \Call{ConvergentsToSolutions}{$q$, $a_1$, $N$, $n$}
	\EndFor
\EndIf
\EndProcedure
\Statex
\Procedure{ConvergentsToSolutions}{$q$, $a_1$, $N$, $n$}
\If{\Call{IsSolution}{$q$, $a_1$, $n$}}
	\If{$n>1$}
		\State \textbf{print} $n$
	\EndIf
	\ForAll{$p\in\msetb{3,\ldots,\msetf{\frac{N}{n}}} : p\text{ prime}$}
		\State \Call{ConvergentsToSolutions}{$q$, $a_1$, $N$, $pn$}
	\EndFor
\EndIf
\EndProcedure
\end{algorithmic}
\end{algorithm}

We implemented Algorithm~\ref{mc:alg:maxcurves} in PARI/GP~\cite{pari-2.9.3}
for pairs $q,a_1$ corresponding to isogeny classes of ordinary elliptic curves,
that is $q$ is a prime power, $\msetv{a_1}\leq 2\sqrt{q}$ and $\gcd\msetp{q,
a_1}=1$. The upper bounds on the degree $n$ in Table~\ref{mc:tbl:qfnq} combined
with Proposition~\ref{mc:prop:apprconv} show that approximating $\frac{\theta}
{\pi}$ up to an error of at most $10^{-15}$ is sufficient to compute the
relevant convergents of $\frac{\theta}{\pi}$. The execution time of the
function \textsc{IsSolution} is reduced by verifying the necessary condition in
Lemma~\ref{mc:lem:logupperbound} before computing $a_n$.

Using our program we computed the triples $\msetp{q,a_1,n}$ with $q<10^6$ a
prime power, $\msetv{a_1}\leq 2\sqrt{q}$, $\gcd\msetp{q,a_1}=1$ and $n>1$ such
that $-a_n=\msetf{2\sqrt{q}^n}$. All triples have $n=3$ or $n=5$, except for
$\msetp{2,1,13}$ and $\msetp{5,1,7}$. The triples with $n=3$ and $q<10^3$ are
listed in Table~\ref{mc:tbl:opn3} and the triples with $n=5$ and $q<10^6$ are
listed in Table~\ref{mc:tbl:opn5}. Based on these results we expect that the
cases $n=3$ and $n=5$ occur infinitely often, whereas the cases $n\geq 7$
happen at most finitely many times.

\begin{table}[t]
\begin{center}
\caption{The list of all pairs $q,a_1$ with $q<10^3$ a prime power, $\msetv{a_1}
\leq 2\sqrt{q}$ and $\gcd\msetp{q,a_1}=1$ such that $-a_3=\msetf{2\sqrt{q}^3}$.}
\begin{tabular}{cc|cc|cc|cc|cc|cc}
$q$ & $a_1$ & $q$ & $a_1$ & $q$ & $a_1$ & $q$ & $a_1$ & $q$ & $a_1$ & $q$ & $a_1$ \\
\hline
 2 & 1 &  37 &  6 & 103 & 10 & 229 & 15 & 479 & 22 & 787 & 28 \\
 3 & 2 &  47 &  7 & 167 & 13 & 257 & 16 & 487 & 22 & 839 & 29 \\
 5 & 2 &  61 &  8 & 173 & 13 & 293 & 17 & 571 & 24 & 967 & 31 \\
 8 & 3 &  67 &  8 & 193 & 14 & 359 & 19 & 577 & 24 & & \\
11 & 3 &  79 &  9 & 197 & 14 & 397 & 20 & 673 & 26 & & \\
17 & 4 &  83 &  9 & 199 & 14 & 401 & 20 & 677 & 26 & & \\
23 & 5 &  97 & 10 & 223 & 15 & 439 & 21 & 727 & 27 & & \\
27 & 5 & 101 & 10 & 227 & 15 & 443 & 21 & 733 & 27 & &
\end{tabular}
\label{mc:tbl:opn3}
\end{center}
\end{table}

\begin{table}[t]
\begin{center}
\caption{The list of all pairs $q,a_1$ with $q<10^6$ a prime power, $\msetv
{a_1}\leq 2\sqrt{q}$ and $\gcd\msetp{q,a_1}=1$ such that $-a_5=\msetf{2\sqrt{q}
^5}$.}
\begin{tabular}{cc|cc|cc}
$q$ & $a_1$ & $q$ & $a_1$ & $q$ & $a_1$ \\
\hline
 2 & -1 &   128 &  -7 &  10399 &  165 \\
 3 & -1 &   317 & -11 &  22159 &  -92 \\
11 & -2 &  2851 & -33 & 122147 & -216 \\
23 & -3 &  8807 & -58 & 192271 & -271 \\
31 &  9 & 10391 & -63 & 842321 & 1485
\end{tabular}
\label{mc:tbl:opn5}
\end{center}
\end{table}

\subsection{Upper bound on the cardinality}
\label{mc:ssec:ubcard}

In this subsection we determine an upper bound on $q$ and conclude:

\begin{theorem}
There exist only finitely many ordinary pairs $q,a_1$ such that $-a_n=\msetf
{2\sqrt{q}^n}$ for some $n\geq 13$.
\end{theorem}

Since Proposition~\ref{mc:prop:laurentbound} also gives an upper bound on the
degree $n$ independent of $q$, the theorem is an immediate consequence of:

\begin{proposition}
Let $n\geq 13$ be an integer. There exists a constant $q_n$ such that if $-a_n=
\msetf{2\sqrt{q}^n}$ for some integers $q,a_1$ with $q\geq 2$ and $\msetv{a_1
}\leq 2\sqrt{q}$, then $q\leq q_n$ or the pair $q,a_1$ is supersingular.
\end{proposition}

The \emph{height} of an algebraic number $\beta$ is defined as the maximum of
the absolute value of the coefficients of the minimal polynomial of $\beta$
over $\mzz$.

\begin{proof}
Assume that $q,a_1$ are integers with $q\geq 2$ and $\msetv{a_1}\leq 2\sqrt
{q}$ such that $-a_n=\msetf{2\sqrt{q^n}}$ for some $n$, then $\msetv{\beta^n+1}<
\frac{1}{\sqrt[4]{q}^n}$ by Lemma~\ref{mc:lem:translation}. Moreover assume that
the pair $q,a_1$ is ordinary, that is $\beta$ is not a root of unity.

Observe that $\beta^n+1=\prod_{i=1}^n\msetp{\beta-\zeta_{2n}^{2i+1}}$. Let $i_0$
be an integer such that
\begin{displaymath}
\msetv{\beta-\zeta_{2n}^{2i_0+1}}=\min_i\msetv{\beta-\zeta_{2n}^{2i+1}},
\end{displaymath}
which determines $i_0$ uniquely modulo $n$ because $\beta$ is not a root of
unity. Since
\begin{displaymath}
\msetv{\beta-\zeta_{2n}^{2i+1}}
\geq\min\msetb{\msetv{\zeta_{2n}^{2i_0}-\zeta_{2n}^{2i+1}},
\msetv{\zeta_{2n}^{2i_0+2}-\zeta_{2n}^{2i+1}}}>0,
\end{displaymath}
for all $i\not\equiv i_0\!\!\mod n$, there exists a positive constant $c_n$ such
that
\begin{displaymath}
\msetv{\beta^n+1}
\geq c_n\msetv{\beta-\zeta_{2n}^m}
\geq c_n\msetv{\frac{a_1}{2\sqrt{q}}-\cos\msetp{\frac{m\pi}{n}}}
\end{displaymath}
with $m=2i_0+1$. Let $\varepsilon>0$. According to \cite[Theorem~2.7]
{bugeaud2004} there exists an ineffective constant $c_0^\prime$ depending on
$\cos\msetp{\frac{m\pi}{n}}$ and $\varepsilon$ such that
\begin{displaymath}
\msetv{\frac{a_1}{2\sqrt{q}}-\cos\msetp{\frac{m\pi}{n}}}
\geq \frac{c_0^\prime}{h^{3+\varepsilon}}
\end{displaymath}
with $h$ the height of $\frac{a_1}{2\sqrt{q}}$. Since there are $n$ possible
values of $m$, the above inequality is also true for some constant $c_0$
depending only on $n$ and $\varepsilon$. The height of $\frac{a_1}{2\sqrt{q}}$
is at most $4q$. Therefore
\begin{displaymath}
\msetv{\beta^n+1}
\geq \frac{c_0c_n}{\msetp{4q}^{3+\varepsilon}}=\frac{c}{q^{3+\varepsilon}}
\end{displaymath}
for some positive constant $c$ depending only on $n$ and $\varepsilon$.

Choose $\varepsilon=\frac{1}{8}$ and $n\geq 13$. The upper and lower bounds on
$\msetv{\beta^n+1}$ imply $c<q^{3+\varepsilon-\frac{n}{4}}$. The right-hand side
converges to zero for $q\rightarrow\infty$, but $c>0$. Hence $q\leq q_n$ for
some constant $q_n$ independent of $\beta$.
\end{proof}

In some sense this proposition is the best possible in terms of $n$, because for
$n=7,9,11$ and $m$ relative prime to $n$ we deduce from \cite[Theorem~2.8]
{bugeaud2004} that there exists a constant $\tilde{c}$ and infinitely many
algebraic numbers $\gamma$ of degree 1 or 2 such that $\msetv{\gamma-\cos\msetp
{\frac{m\pi}{n}}}<\frac{\tilde{c}}{h_\gamma^{3-\varepsilon}}$ where $h_\gamma$
is the height of $\gamma$. If $h_\gamma\sim q$, then this upper bound
is eventually smaller than $\frac{1}{\sqrt[4]{q}^n}$.

\section{Maximal over cubic extensions}
\label{mc:sec:max3}

In this section we prove Theorem~\ref{mc:thm:max3prime}. For the sake of
completeness we also discuss some properties of the case $n=3$. The discussion
is closely related to \cite[Section~2.7]{soomro2013}.

Given a supersingular pair $q,a_1$ such that $\msetv{a_1}\leq 2\sqrt{q}$ and
$-a_3=\msetf{2\sqrt{q}^3}$, then $a_1=-2\sqrt{q}$ or $a_1=\sqrt{q}$ by
Proposition~\ref{mc:prop:qscm}. In this case $q$ must be a square. Since $q$ is
a prime in Theorem~\ref{mc:thm:max3prime}, we only consider ordinary pairs.

Recall the recurrence relation $a_{n+1}=a_1a_n-qa_{n-1}$ with $a_0=2$ mentioned
in the introduction. From this we deduce $a_3=a_1^3-3qa_1$. Therefore
\begin{displaymath}
-a_3=\msetf{2\sqrt{q}^3}
\quad\Longleftrightarrow\quad
0\leq a_1^3-3qa_1+2\sqrt{q}^3<1.
\end{displaymath}
Define the function $f_q:\msets{-2\sqrt{q},2\sqrt{q}}\rightarrow\mrr$ as $x
\mapsto x^3-3qx+2\sqrt{q}^3$. The graph of $f_q$ is shown in Figure~\ref
{mc:fig:max-deg-3}.

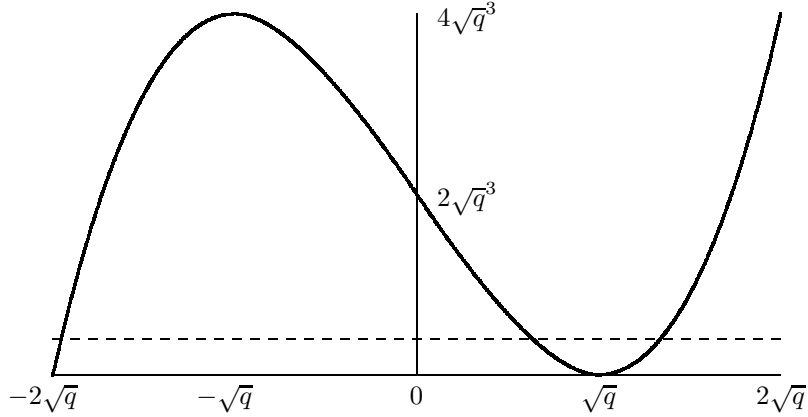
\begin{figure}
\begin{center}
\newlength{\mlw}
\settowidth{\mlw}{$-4\sqrt{q}^3-$}
\newlength{\mlh}
\settoheight{\mlh}{$-4\sqrt{q}^3-$}
\newlength{\mld}
\settoheight{\mld}{$-4\sqrt{q}^3-$}
\newlength{\fw}
\setlength{\fw}{0.9\textwidth}
\newlength{\ifw}
\setlength{\ifw}{\fw}
\addtolength{\ifw}{-\mlw}
\newlength{\ifh}
\setlength{\ifh}{0.5\ifw}
\newlength{\ibm}
\setlength{\ibm}{1.2\mlh}
\addtolength{\ibm}{\mld}
\newlength{\fh}
\setlength{\fh}{\ifh}
\addtolength{\fh}{\ibm}
\addtolength{\fh}{0.5\mlh}
\setlength{\unitlength}{0.5\ifw}
\begin{picture}(\fw,\fh)(-0.5\fw,-\ibm)
\thinlines
\put(-1,0){\line(1,0){2}}
\put(0,0){\line(0,1){1}}
\thicklines
\qbezier(-1,0)(-0.778,1)(-0.5,1)
\qbezier(-0.5,1)(-0.333,1)(0,0.5)
\qbezier(0,0.5)(0.333,0)(0.5,0)
\qbezier(0.5,0)(0.778,0)(1,1)
\thinlines
\multiput(-1,0.1)(0.0494,0){41}{\line(1,0){0.0247}}
\newlength{\lx}
\newlength{\ly}
\newlength{\lw}
\newlength{\lh}
\setlength{\ly}{-1.2\mlh}
\settowidth{\lw}{$-2\sqrt{q}-$}
\setlength{\lx}{-\unitlength}
\addtolength{\lx}{-0.5\lw}
\put(\lx,\ly){$-2\sqrt{q}$}
\settowidth{\lw}{$-\sqrt{q}-$}
\setlength{\lx}{-0.5\unitlength}
\addtolength{\lx}{-0.5\lw}
\put(\lx,\ly){$-\sqrt{q}$}
\settowidth{\lw}{$0$}
\setlength{\lx}{-0.5\lw}
\put(\lx,\ly){$0$}
\settowidth{\lw}{$\sqrt{q}$}
\setlength{\lx}{0.5\unitlength}
\addtolength{\lx}{-0.5\lw}
\put(\lx,\ly){$\sqrt{q}$}
\settowidth{\lw}{$2\sqrt{q}$}
\setlength{\lx}{\unitlength}
\addtolength{\lx}{-0.5\lw}
\put(\lx,\ly){$2\sqrt{q}$}
\setlength{\lx}{0.2\mlw}
\setlength{\ly}{0.5\unitlength}
\settoheight{\lh}{$2\sqrt{q}^3$}
\addtolength{\ly}{-0.5\lh}
\put(\lx,\ly){$2\sqrt{q}^3$}
\setlength{\ly}{\unitlength}
\settoheight{\lh}{$2\sqrt{q}^3$}
\addtolength{\ly}{-0.5\lh}
\put(\lx,\ly){$4\sqrt{q}^3$}
\end{picture}
\end{center}
\caption{The graph of $f_q\msetp{a}=a^3-3qa+2\sqrt{q}^3$.}
\label{mc:fig:max-deg-3}
\end{figure}

\begin{proposition}
Let $q,a_1$ be integers such that $q\geq 3$ and $\msetv{a_1}\leq 2\sqrt{q}$. If
$-a_3=\msetf{2\sqrt{q}^3}$, then $a_1=-\msetf{2\sqrt{q}}$ or $a_1=\msets{\sqrt
{q}}$.
\end{proposition}
\begin{proof}
Notice that $f_q$ is maximal at $x=-\sqrt{q},2\sqrt{q}$ and that $f_q$ is
minimal at $x=-2\sqrt{q},\sqrt{q}$ and
\begin{displaymath}
f_q\msetp{-2\sqrt{q}+1}=\msetp{3\sqrt{q}-1}^2>1
\end{displaymath}
and
\begin{displaymath}
f_q\msetp{\sqrt{q}\pm \frac{1}{2}}=\frac{3}{4}\sqrt{q}\pm\frac{1}{8}>1.
\end{displaymath}
Hence $-2\sqrt{q}\leq a_1<-2\sqrt{q}+1$ or $\sqrt{q}-\frac{1}{2}<a_1<\sqrt{q}+
\frac{1}{2}$, that is $a_1=-\msetf{2\sqrt{q}}$ or $a_1=\msets{\sqrt{q}}$.
\end{proof}

According to the following proposition the case $a_1=-\msetf{2\sqrt{q}}$ is
possible only if the pair $q,a_1$ is supersingular.

\begin{proposition}
Let $q$ be an integer with $q\geq 2$ and $a_1=-\msetf{2\sqrt{q}}$. If $-a_3=\msetf{2
\sqrt{q}^3}$, then $q$ is a square.
\end{proposition}
\begin{proof}
Assume that $q$ is not a square. Let $a=-a_1=\msetf{2\sqrt{q}}$.

The function $f_q$ is strictly monotonically increasing and strictly concave on
the interval $\msetp{-2\sqrt{q},-\sqrt{q}}$, because $\frac{df_q}{dx}=3x^2-3q$
and $\frac{d^2f_q}{dx^2}=6x$ are positive and negative respectively. Let $x_0$
be the intersection between the line $y=1$ and the line through $\msetp{-2\sqrt
{q},0}$ and $\msetp{-2\sqrt{q}+1,f_q\msetp{-2\sqrt{q}+1}}$. Then
\begin{displaymath}
a_1+2\sqrt{q}
<x_0+2\sqrt{q}
=\frac{1}{\msetp{3\sqrt{q}-1}^2}.
\end{displaymath}

Notice that $4q=a^2+b$ with $1\leq b\leq 2a$. Since $\sqrt{1+x}\geq 1+\msetp{
\sqrt{2}-1}x$ for $0\leq x\leq 1$,
\begin{displaymath}
a_1+2\sqrt{q}
=a\msetp{-1+\sqrt{1+\frac{b}{a^2}}}
\geq \msetp{\sqrt{2}-1}\frac{b}{a}
\geq \frac{\sqrt{2}-1}{a}
\geq \frac{\sqrt{2}-1}{\sqrt{q}}.
\end{displaymath}

Combining the upper and lower bounds on $-a+2\sqrt{q}$ yields
\begin{displaymath}
0>\msetp{\sqrt{2}-1}\msetp{3\sqrt{q}-1}^2-\sqrt{q},
\end{displaymath}
but the right-hand side is positive by construction. Contradiction.
\end{proof}

We recall a sufficient condition on $q$ such that $-a_3=\msetf{2\sqrt{q}^3}$
for $a_1=\msets{\sqrt{q}}$. It is~\cite[Proposition~2.7.1]{soomro2013} with a
different proof.

\begin{proposition}[Soomro]
\label{mc:prop:soomro}
If $q=a_1^2+b$ with integers $a_1,b$ such that $a_1\geq 2$ and $\msetv{b}\leq
\sqrt{a_1}$, then $-a_3=\msetf{2\sqrt{q}^3}$.
\end{proposition}
\begin{proof}
Let $0<\epsilon\leq\frac{1}{3}$. Consider the function
\begin{displaymath}
g_\epsilon\msetp{x}
=1+\frac{3}{2}x+\frac{3}{8}\msetp{1+\epsilon}x^2-\sqrt{1+x}^3
\end{displaymath}
and compute $\frac{dg_\epsilon}{dx}=\frac{3}{2}+\frac{3}{4}\msetp{1+\epsilon}x-
\frac{3}{2}\sqrt{1+x}$ and $\frac{d^2g_\epsilon}{dx^2}=\frac{3}{4}\msetp{1+
\epsilon}-\frac{3}{4}\sqrt{1+x}^{-1}$. The function $g_\epsilon$ has extrema in
$x=-\frac{4\epsilon}{\msetp{1+\epsilon}^2}$ and $x=0$. The former is a maximum
and the latter is a minimum. Let $x_\epsilon$ the unique zero of $g_\epsilon$
such that $-1\leq x_\epsilon<-\frac{4\epsilon}{\msetp{1+\epsilon}^2}$. Hence
for all $x>x_\epsilon$ and $x\neq 0$
\begin{displaymath}
\sqrt{1+x}^3
<1+\frac{3}{2}x+\frac{3}{8}\msetp{1+\epsilon}x^2.
\end{displaymath}

Define $x=\frac{b}{a_1^2}$. Notice that
\begin{displaymath}
f_q\msetp{a_1}
=-2{a_1}^3-3ba_1+2\sqrt{a_1^2+b}^3
=2a_1^3\msetp{-1-\frac{3}{2}x+\sqrt{1+x}^3}
\end{displaymath}
is minimal on $\msetp{x_\epsilon,\infty}$ for $x=0$. If $x>x_\epsilon$ and $x
\neq 0$, then
\begin{displaymath}
0
\leq f_q\msetp{a_1}
=2a_1^3\msetp{-1-\frac{3}{2}x+\sqrt{1+x}^3}
<\frac{3}{4}\msetp{1+\epsilon}\frac{b^2}{a_1}.
\end{displaymath}
Observe that if $b=0$ (or $x=0$) then $f_q\msetp{a_1}=0$.

Assume that $\epsilon=\frac{1}{3}$ and $\msetv{b}\leq\sqrt{a_1}$, then $x\geq
-\sqrt{a_1}^{-3}>-1=x_\epsilon$ and $0\leq f_q\msetp{a_1}<\frac{3}{4}\msetp{1+
\epsilon}=1$. Hence $-a_3=\msetf{2\sqrt{q}^3}$.
\end{proof}

A closer look at the proof tells us that in the proposition above the
constraint $b^2\leq a_1$ can be replaced by $b^2\leq \frac{4}{3}\frac{1}{1+
\epsilon}a_1$ at the expense of introducing a lower bound on $a_1$ in terms of
$\epsilon$.

Before proving Theorem~\ref{mc:thm:max3prime}, let us motivate that it is a
non-trivial statement. Let $q,a_1$ be a pair such that $q$ is an odd prime and
$\msetv{a_1}\leq 2\sqrt{q}$. Then $a_1=\msets{\sqrt{q}}$ and the proposition
above suggests that $q=a_1^2+b$ with $b^2\leq ca_1$ for some positive constant
$c$. However the primes of this form have Dirichlet density zero, see \cite
[Remark~2.7.2]{soomro2013}. Hence it is unlikely to find such primes.

The idea of the proof is to reduce the problem to a question on Gaussian primes
in a small sector of the plane and apply \cite[Theorem~1]{harman2001}.

\begin{proof}[Proof of Theorem~\ref{mc:thm:max3prime}]
Consider the set
\begin{displaymath}
S_1
=\msetb{\msetp{a,b}\in\mzz^2:p=a^2+b\text{ prime},0<a,\msetv{b}\leq\sqrt{a}}
\end{displaymath}
and the subset $S_2=\msetb{\msetp{a,b}\in S_1:b\text{ square}}$. The set
$S_2$ corresponds to
\begin{displaymath}
S_3
=\msetb{\msetp{a,c}\in\mzz^2:p=a^2+c^2\text{ prime},0<a,0\leq c\leq\sqrt[4]
{a}}.
\end{displaymath}
Define for $\theta>0$
\begin{displaymath}
S_4\msetp{\theta}=\msetb{\msetp{a,c}\in\mzz^2:p=a^2+c^2\text{ prime},0<a,0\leq
c<p^\theta}
\end{displaymath}
and write $S_4\msetp{\theta}=S_5\msetp{\theta}\cup S_6\msetp{\theta}$ with $S_5
\msetp{\theta}=\msetb{\msetp{a,c}\in S_4\msetp{\theta}:a\geq p^{4\theta}}$ and
$S_6\msetp{\theta}=\msetb{\msetp{a,c}\in S_4\msetp{\theta}:a<p^{4\theta}}$.
Observe that $S_5\msetp{\theta}\subset S_3$. If $\theta<\frac{1}{8}$, then the
set $S_6\msetp{\theta}$ is finite, because $p=a^2+c^2<p^{8\theta}+p^{2\theta}$
and
\begin{displaymath}
\lim_{p\rightarrow\infty}\msetp{p^{8\theta-1}+p^{2\theta-1}}=0.
\end{displaymath}

The set $S_4\msetp{0.119}$ is infinite by \cite[Theorem~1]{harman2001} and
$0.119<\frac{1}{8}$. Hence the sets $S_5\msetp{0.119}\subset S_3$ and $S_2
\subset S_1$ are also infinite. If $p=a_1^2+b\in S_1$, then $\msetv{a_1}\leq
2\sqrt{q}$ and $-a_3=\msetf{2\sqrt{q}^3}$ by Proposition~\ref{mc:prop:soomro}.
\end{proof}

\bibliographystyle{plainurl}
\bibliography{bibliography}

\end{document}